\documentclass[12pt,reqno]{amsart}
\usepackage[utf8]{inputenc}
\usepackage[T1]{fontenc}
\usepackage{mathptmx}
\usepackage{amssymb}

\usepackage{paralist}

\usepackage{enumitem}

\usepackage[colorlinks=true, plainpages=false]{hyperref}
\usepackage{cleveref}

\usepackage{tikz}
\usetikzlibrary{calc}

\usepackage{color}

\newtheorem{theorem}{Theorem}[section]
\newtheorem*{thm}{Theorem}
\newtheorem{lemma}[theorem]{Lemma}
\newtheorem{corollary}[theorem]{Corollary}
\newtheorem{proposition}[theorem]{Proposition}
\newtheorem{conjecture}[theorem]{Conjecture}
\newtheorem{question}[theorem]{Question}

\theoremstyle{definition}
\newtheorem{remark}[theorem]{Remark}
\newtheorem{definition}[theorem]{Definition}
\newtheorem{example}[theorem]{Example}

\numberwithin{equation}{section}

\def\ZZ{{\mathbb Z}}
\def\NN{{\mathbb N}}
\def\KK{{\mathbb K}}

\def\sdepth{\operatorname{sdepth}}
\def\depth{\operatorname{depth}}
\DeclareMathOperator{\lcm}{lcm}
\DeclareMathOperator{\spdim}{spdim}
\DeclareMathOperator{\pdim}{pdim}

\newcommand{\set}[1]{\{#1\}}
\newcommand{\with}{\,:\,}
\newcommand{\subsetIJ}{\subsetneq}
\newcommand{\defa}{:=}
\newcommand{\ord}[2]{\mathrm{\mathop{ord}}_{#1}(#2)}

\newcommand{\qq}[1]{``#1''}

\newcommand{\gen}[1]{{#1}_\varepsilon}

\newcommand{\ve}{\varepsilon}
\newcommand{\xf}{\mathbf{x}}

\newcommand{\slat}{L}
\newcommand{\lat}{\overline{L}}

\renewcommand\>{\rangle}
\newcommand\<{\langle}

\begin{document}

\title{Stanley depth and the lcm-lattice}

\author{Bogdan Ichim}

\address{Simion Stoilow Institute of Mathematics of the Romanian Academy, Research Unit 5, C.P. 1-764,
014700 Bucharest, Romania} \email{bogdan.ichim@imar.ro}

\author{Lukas Katth\"an}

\address{Goethe-Universit\"at Frankfurt, FB Informatik und Mathematik, 60054 Frankfurt am Main, Germany}
\email{katthaen@math.uni-frankfurt.de}

\author{Julio Jos\'e Moyano-Fern\'andez}

\address{Universitat Jaume I, Campus de Riu Sec, Departamento de Matem\'aticas \& Institut Universitari de Matem\`atiques i Aplicacions de Castell\'o, 12071
Caste\-ll\'on de la Plana, Spain} \email{moyano@uji.es}


\thanks{The first author was partially supported  by project  PN-II-RU-TE-2012-3-0161, granted by the Romanian National Authority for Scientific Research,
CNCS - UEFISCDI. The second author was partially supported by
the German Research Council DFG-GRK~1916. The third author was partially supported by the Spanish Government Ministerio de Econom\'ia y Competitividad (MINECO), grants MTM2012-36917-C03-03 and MTM2015-65764-C3-2-P, as well as by Universitat Jaume I, grant P1-1B2015-02.}

\begin{abstract}
In this paper we show that the Stanley depth, as well as the usual depth, are essentially determined by the lcm-lattice.
More precisely, we show that for quotients $I/J$ of monomial ideals $J\subset I$, both invariants behave monotonic with respect to certain maps defined on their lcm-lattice.
This allows simple and uniform proofs of many new and known results on the Stanley depth.
In particular, we obtain a generalization of our result on polarization presented in \cite{IKM}.
We also obtain a useful description of the class of all monomial ideals with a given lcm-lattice, which is independent from our applications to the Stanley depth.
\end{abstract}

\maketitle

\section{Introduction}

Let $\KK$ be a field, $S$ a $\NN^n$-graded $\KK$-algebra and $M$ a finitely generated $\ZZ^n$-graded $S$-module.
The \emph{Stanley depth} of $M$, denoted $\sdepth_S M$, is a combinatorial invariant of $M$ which was introduced by Apel in \cite{A1} and has attracted the attention of many researchers \cite{ HP,HVZ, BHKTY, OY, counterexample}.
We refer the reader to the survey of Herzog \cite{H} for an introduction to this subject.

One line of research on the Stanley depth was motivated by a conjecture of Stanley from 1982 \cite[Conjecture 5.1]{St},
which states that $\depth_S M \leq \sdepth_S M$, see also \cite[Remark 5.2]{G} and \cite[p.~149]{stanley79}.
In the following, we refer to this inequality as \emph{Stanley's inequality}.
Thus, a mayor aim of this study is to establish relations between the depth and the Stanley depth.
At a first glance, one might not expect any deep connection between them, at these invariants seem to be very different in nature.
On the one hand, we have the depth, which is an algebraic invariant (homological in nature), and on the other hand we have the Stanley depth, which is a purely combinatorial invariant.
Nevertheless, several parallel results for the depth and the Stanley depth have been found, see for example \cite{R2, Ci, HJZ, Is, SF}.
A counterexample to the original Stanley conjecture was recently given by Duval, Goeckner, Klivans and Martin in \cite{counterexample}. However, there still seems to be a deep and interesting connection between these two invariants, which is yet to be fully understood.

The Stanley depth is defined in terms of certain combinatorial decompositions of the module $M$, which are called Stanley decompositions.
It is also worth mentioning that these Stanley decompositions have a separate life in applied mathematics.
Sturmfels and White \cite{Sturmfels1991} have shown that Stanley decompositions may be used to describe finitely generated graded algebras, for example rings of invariants under some group action.
Recently this has found applications in the normal form theory for systems of differential equations (see Murdock \cite{Murdock2002}, Murdock and Sanders \cite{Murdock2007}, Sanders \cite{Sanders2007}).

Most of the research on the Stanley depth concentrates on the particular case of a module of the form $I/J$ for two monomial ideals $J \subsetIJ I$ in
the polynomial ring $S=\KK[X_1, \ldots , X_n]$.
In this paper we will also work in this setting.
For the sake of simplicity, we restrict this introduction to the case of modules of the form $S/I$ (most of the results are later proven in a more general setup).

The lcm-lattice $L_I$ of an ideal $I \subset S$ is the lattice of all least common multiples of subsets of the (minimal) generators of $I$, ordered by divisibility, see Gasharov et al. \cite{GPW}. It is a finite atomistic lattice that is known to encode a lot of information about $I$.
In particular, it encodes the structure of the minimal free resolution of $S/I$ over $I$ and thus determines the Betti numbers and the projective dimension of $S/I$ \cite[Theorem 3.3]{GPW}.
More precisely, what is shown in \cite{GPW} is the following: Let $I \subset S$ and $I' \subset S'$ be two monomial ideals.
Then, given a free resolution of $S/I$ and a surjective join-preserving map $L_I \rightarrow L_{I'}$ which is bijective on the atoms, one can construct a free resolution of $S'/I'$
by a certain relabeling procedure. In particular, the projective dimension of $S'/I'$ is bounded above by the projective dimension of $S/I$.
In this paper we obtain the analogous statement for the Stanley depth.
\begin{thm}[Corollary of \Cref{thm:lcmmap}]
	Let $I \subset S$ and $I' \subset S'$ be two proper monomial ideals in two polynomial rings in $n$ resp. $n'$ variables.
	If there exists a surjective join-preserving map $\delta: L_I \rightarrow L_{I'}$, such that $\delta^{-1}(\hat{0}) = \set{\hat{0}}$, then
	\[ n - \sdepth_S S/I \geq n' - \sdepth_{S'} S'/I'. \]
\end{thm}
By $\hat{0}$ we denote the minimal elements of the lattices.
In view of this result, we define the \emph{Stanley projective dimension}, $\spdim_S M$, analogously to the Stanley depth (cf. Definition \ref{def:spdim}).
In particular, it easily follows that $\spdim_S S/I = n - \sdepth_S S/I$.

\begin{example}\label{ex:main}
	For $S = \KK[x,y]$ and $S' = \KK[x,y,z,v]$ consider the two ideals
	\begin{align*}
	I &:= \<x^3,x^2y,xy^2,y^3\> \subset S \ \ \text{and}\\
	I' &:= \<yzv,xzv,xy^2v,x^2y^2z\>\subset S'.
	\end{align*}
	Their lcm-lattices are depicted in \Cref{fig:exInt}. We define a map $\delta: L_{I} \to L_{I'}$ by setting $\delta(x^2y^2) = xy^2zv$, $\delta(xy^3) = \delta(x^2y^3) = x^2y^2zv$, and every other monomial in $L_{I}$ is mapped to the monomial in $L_{I'}$ which is at the same place in \Cref{fig:exInt}.
	
	This map is join-preserving and surjective, so \Cref{thm:lcmmap} applies.
	It is clear that $\sdepth_S S / I = 0$ and so $\spdim_S S/I = 2 - \sdepth_S S / I = 2$.
	It follows that $\spdim_{S'} S' / I'= 4 - \sdepth_{S'} S' / I' \leq 2$, or equivalently that
	$\sdepth_{S'} S' / I' \geq 2$.
	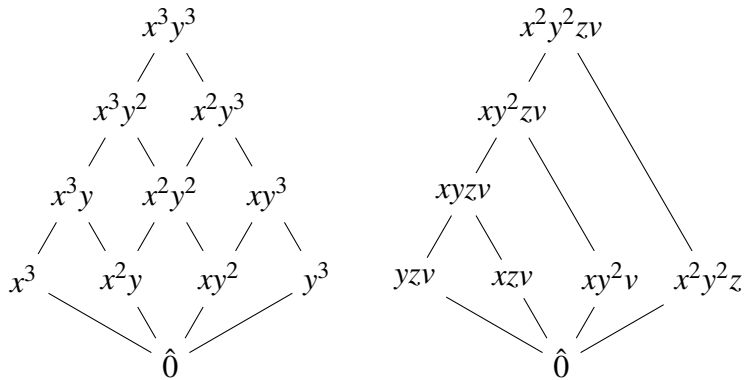
\begin{figure}[h]
	\centering
	\begin{tikzpicture}[scale=1.3,xscale=1.0]
	
	\begin{scope}[]
	\path (0,0)
		node (u1) at +(0,0) {$x^3$}
		node (u2) at +(1,0) {$x^2y$}
		node (u3) at +(2,0) {$xy^2$}
		node (u4) at +(3,0) {$y^3$};
	\path (0,0) (60:1)
		node (m1) at +(0,0) {$x^3y$}
		node (m2) at +(1,0) {$x^2y^2$}
		node (m3) at +(2,0) {$xy^3$};
	\path (0,0) (60:2)
		node (mm1) at +(0,0) {$x^3y^2$}
		node (mm2) at +(1,0) {$x^2y^3$};
	\path (0,0) (60:3)
		node (t) at +(0,0) {$x^3y^3$};
	\path (1,0)
		node (bot) at +(-60:1) {$\hat{0}$};
	
	\draw (u1)-- (m1) -- (mm1) -- (t);
	\draw (u2)-- (m2) -- (mm2);
	\draw (u3)-- (m3);
	\draw (u4)-- (m3) -- (mm2) -- (t);
	\draw (u3)-- (m2) -- (mm1);
	\draw (u2)-- (m1);
	
	\draw (u1) -- (bot) -- (u2);
	\draw (u3) -- (bot) -- (u4);
	\end{scope}
	
	\begin{scope}[xshift=4cm]
		\path (0,0)
		node (u1) at +(0,0) {$yzv$}
		node (u2) at +(1,0) {$xzv$}
		node (u3) at +(2,0) {$xy^2v$}
		node (u4) at +(3,0) {$x^2y^2z$}
		node (l1) at +(60:1) {$xyzv$}
		node (l2) at +(60:2) {$xy^2zv$}
		node (l3) at +(60:3) {$x^2y^2zv$};
		\path (1,0)
		node (bot) at +(-60:1) {$\hat{0}$};
		
		\draw (u1)-- (l1) -- (l2) -- (l3);
		\draw (u2)-- (l1);
		\draw (u3)-- (l2);
		\draw (u4)-- (l3);
		
		\draw (u1) -- (bot) -- (u2);
		\draw (u3) -- (bot) -- (u4);
	\end{scope}
	
	\end{tikzpicture}
	\caption{The lcm-lattices of the two monomial ideals from \Cref{ex:main}}
	\label{fig:exInt}
	\end{figure}
\end{example}

Theorem \ref{thm:lcmmap} has a number of important consequences.
First of all, it shows that two ideals with isomorphic lcm-lattices have the same Stanley projective dimension.
Thus, this invariant is determined by the isomorphism type of the lcm-lattice.
In particular, the lcm-lattice of an ideal is invariant under polarization. Hence Theorem \ref{thm:lcmmap} generalizes the main result of \cite{IKM}, where we showed that the Stanley projective dimension is invariant under polarization.

Next, we present a simple and uniform proof for upper bounds on the Stanley projective dimension (i.e.~lower bounds on the Stanley depth) in terms of the number of generators in Proposition \ref{prop:boundI}.
We also characterize the extremal case and prove that Stanley's inequality holds for ideals with $\pdim_S S/I = k - 1$, where $k$ is the number of generators of $I$.
As another application we study generic deformations in Proposition \ref{prop:gen} and the forming of colon ideals in Proposition \ref{prop:colon}. The latter allows us to give in Corollary \ref{cor:ass} a uniform proof that both depth and Stanley depth are bounded by the dimensions of the associated prime ideals.
Moreover, in Proposition \ref{prop:singledegree} we show that for studying the Stanley depth one may always assume that the ideal under consideration is generated in a single degree.

We further identify some operations on ideals, e.g. passing to the radical, that yield surjective join-preserving maps on the lcm-lattice, so we obtain inequalities for the Stanley projective dimension in these cases.
As all our proofs rest on Theorem \ref{thm:lcmmap} and we showcase an analogous result for the usual projective dimension (Theorem \ref{thm:lcmdepth}), we obtain the same bounds as for the usual projective dimension. While these results are well-known, it is relevant that we obtain \emph{uniform} proofs for both depth and Stanley depth, thus explaining the observed parallel behavior.

In the way of studying the relation of ideals to their lcm-lattices, we also get a result of independent interest. In Theorem \ref{thm:lcmgeneral} we give a complete description of the class of all monomial ideals with a given lcm-lattice.
This result also allows the easy construction of monomial ideals with a prescribed lcm-lattice, which we consider very useful for the study of examples.
Theorem \ref{thm:lcmgeneral} extends results obtained in \cite{M} to the (more general) case of not necessarily atomistic lattices, which is needed for our applications to both depth and Stanley depth.

Finally, the fact that both the projective dimension and the Stanley projective dimension are determined by the lcm-lattice allows us further to formulate open questions about the depth and the Stanley depth completely in terms of finite lattices.
So one can try to apply notions and techniques from this field to approach these questions.
In the last section we indicate some of these ideas.
In particular, this enables us to reduce the study of \emph{infinitely} many monomial ideals to \emph{finitely} many \emph{finite} lattices.
This paves the way to computations.
In several computational experiments, we have classified all lcm-lattices of ideals $I$ with up to five generators and found that
the projective dimension and the Stanley projective dimension of $S/I$ coincide for these lattices; this is presented in \cite{IKM3}.

In another follow-up paper \cite{K}, the second author applied Theorem \ref{thm:lcmmap} to show that many questions about the Stanley depth can be reduced to a very special class of ideals.
In particular, Stanley's inequality holds for ideals with up to seven generators and for quotients of the polynomial ring by ideals with up to six generators.

\section{Preliminaries}\label{sec:pre}

\subsection{Finite lattices and semilattices}
Let us recall some definitions and facts about finite lattices and semilattices.
We refer the reader to \cite{DP} for more background information.
A \emph{join-semilattice} $L$ is a partially ordered set $(L,\leq)$ such that, for any $P,Q \in L$, there is a unique least upper bound $P\vee Q$ called the join of $P$ and $Q$.
A \emph{lattice} is a join-semilattice $L$ with the additional property that for any $P,Q \in L$, there is a unique greatest lower bound $P \wedge Q$ called the meet of $P$ and $Q$.

Every finite join-semilattice has a unique maximal element $\hat{1}$.
Moreover, a finite join-semilattice is a lattice if and only if it has a minimal element.
So we can associate to every finite join-semilattice $\slat$ a canonical lattice $\lat := \slat \cup \set{\hat{0}}$ by adjoining a minimal element $\hat{0}$.
All lattices and semilattices in the sequel will be assumed to be finite.

We say that an element $a \in L$ \emph{covers} another element $b \in L$, if $b < a$ and there exists no other element $c \in L$, such that $b < c < a$.
An element is called an \emph{atom} if it covers the minimal element $\hat{0}$ in $\lat$. Equivalently, the atoms are the minimal elements of $\slat$ (in the sense that there are no smaller elements).
We call $L$ \emph{atomistic}, if every element can be written as a join of atoms.

A \emph{meet-irreducible} element is an element which is covered by exactly one other element. This terminology is justified by noting that $a$ is meet-irreducible if and only if $a=b \wedge c$ implies $a=b$ or $a=c$ for $b,c \in \lat$ where the meet is taken in $\slat$.
A \emph{join-preserving map} $\delta: L \to L'$ is a map with $\delta(a \vee b)=\delta(a) \vee \delta(b)$ for all $a,b \in L$. Note that every join-preserving map preserves the order.

\subsection{The lcm-lattice and lcm-closed subsets}
\newcommand{\Mon}[1]{\mathop{\mathrm{Mon}}(#1)}
\newcommand{\LCMs}[1]{L(#1)}
\newcommand{\LCMl}[1]{{#1}_{\hat{0}}}
\renewcommand{\LCMl}[1]{\overline{#1}}

Let $S = \KK[X_1,\dotsc, X_n]$ be a polynomial ring.
A \emph{monomial} $m \in S$ is a product of powers of variables of $S$. In particular, $1_\KK$ is a monomial, but $0_\KK$ is not.
We write $\Mon{S}$ for the set ot monomials of $S$. Note that $\Mon{S}$ forms a $\KK$-basis of $S$.

Recall from \cite{GPW} that the \emph{lcm-lattice} $L_I$ of a monomial ideal $I \subset S$ is the set of all least common multiples of subsets of the minimal set of generators of $I$, together with a minimal element $\hat{0}$ which is usually identified with $1_\KK$ and regarded as the lcm of the empty set.
For our scope, we need to modify this notion in several ways.
First, we need to consider non-minimal generating sets, second we need a reasonable replacement of the lcm-lattice for a pair $J \subsetneq I$ of ideals, and finally we want that our modified definition yields isomorphic lattices for all principal ideals, including the unit ideal.
To this end we give the following definition.
\begin{definition}
	We call a finite set $G \subset \Mon{S}$ of monomials \emph{lcm-closed}, if the least common multiple (lcm) of every non-empty subset of $G$ is also contained in $G$.
	
	The \emph{lcm-closure} of a finite set $G \subset \Mon{S}$, denoted by $\LCMs{G} \subset \Mon{S}$, is defined as the set of all monomials that can be obtained as the least common multiple (lcm) of some non-empty subset of $G$
\end{definition}

Note that $G \subset \LCMs{G}$ and that $G = \LCMs{G}$ if and only if $G$ is lcm-closed.
Every lcm-closed set $G$ can be regarded as a join-semilattice, where the order is given by divisibility and the join is the lcm.
We will often consider the associated lattice $\LCMl{G} := G \cup \set{\hat{0}}$ of an lcm-closed set $G$, where $\hat{0}$ is an additional minimal element.
The element $\hat{0}$ could be regarded as the lcm of the empty set, but we do not identify $\hat{0}$ with $1_\KK$.

Note that if $I \subsetneq S, I \neq \<0\>$ is a proper monomial ideal, then $L_I = \LCMl{\LCMs{G(I)}}$, where $G(I)$ is a minimal generating set of $I$.

\begin{remark}\label{rem:semilattice}
\begin{asparaenum}
	\item Following \cite{GPW}, for $I = S$ we get $L_I=\set{1_\KK}$, but $\LCMl{\LCMs{G(I)}} = \set{\hat{0}, 1_\KK}$.
	While this difference is minor, we think that the latter is in fact a more convenient definition of the lcm-lattice of $S$. For example, the lcm-lattice of $S$ should be isomorphic to the lcm-lattice of a principal ideal.
	
	\item The associated lattice $\LCMl{G}$ of an lcm-closed set $G \subset \Mon{S}$ is atomistic if and only if the minimal elements of $G$ form the minimal set of generators of some monomial ideal (in fact, $\<G\>$).
	So in general, $\LCMl{G}$ could be regarded as an lcm-lattice associated to a not necessarily minimal set of generators of $\<G\>$.
	The reason why we consider non-minimal generating sets is that we are going to consider maps of lcm-lattices.
	Even if we start with a minimal set of generators of some monomial ideal, its image might not be minimal anymore.
\end{asparaenum}
\end{remark}

\subsection{Stanley depth and maps changing it}
Consider the polynomial ring $S$ endowed with the multigraded structure.
Let $M$ be a finitely generated graded $S$-module, and let $\lambda$ be a homogeneous element in $M$.
Let $Z \subset \set{X_1, \ldots , X_n}$ be a subset of the set of indeterminates of $S$. The $\KK[Z]$-submodule $\lambda \KK[Z]$ of $M$ is called a \emph{Stanley space} of $M$ if $\lambda \KK[Z]$ is free (as $\KK[Z]$-submodule).
A \emph{Stanley decomposition} of $M$ is a finite family
\[
\mathcal{D}=(\KK[Z_i],\lambda_i)_{i\in \mathcal{I}}
\]
in which $Z_i \subset \{X_1, \ldots ,X_n\}$ and $\lambda_i\KK[Z_i]$ is a Stanley space of $M$ for each $i \in \mathcal{I}$ with
\[
M = \bigoplus_{i \in \mathcal{I}} \lambda_i\KK[Z_i]
\]
as a multigraded $\KK$-vector space. This direct sum carries the structure of an $S$-module and has therefore a well-defined depth. The \emph{Stanley depth} $\sdepth\ M$ of $M$ is defined to be the maximal depth of a Stanley decomposition of $M$.

In the same fashion we introduce the following definition.
\begin{definition}\label{def:spdim}
The \emph{Stanley projective dimension} $\spdim_S M$ of $M$ is the minimal projective dimension of a Stanley decomposition of $M$.
\end{definition}
Note that $\spdim_S M = n - \sdepth_S M$ by the Auslander-Buchsbaum formula.
While this definition might seem redundant, it turns out that our results (for example, see Theorem \ref{thm:lcmmap}) are more naturally stated in terms of the Stanley projective dimension.
Further, Stanley's inequality is equivalent to the following:
\begin{equation}\label{conj:stanley}
\pdim\ M \geq \spdim\ M. \tag{$\star$}
\end{equation}
The original Stanley conjecture \cite{St} stated that \eqref{conj:stanley} holds for all finitely generated modules. As mentioned above, this was recently disproved by Duval et al. \cite{counterexample}.

In the proof of our main result we use a certain type of poset maps which was first introduced in \cite{IKM}.
Before we recall the definition, let us introduce a notation.
For $a, g \in \NN^n$ with $a \leq g$, we define
\[ \rho_g(a) := \#\set{j \in [n] \with a_j = g_j}.\]

\begin{definition}\cite[Definition 3.1]{IKM} \label{def:1}
Let $\ell \in \ZZ$ and $n, n' \in \NN$. A monotonic map $\phi: \NN^{n} \to \NN^{n'}$ is said to \emph{change the Stanley depth} by $\ell$ with respect
to $g \in \NN^{n}$ and $g' \in\NN^{n'}$, if it satisfies the following two conditions:
\begin{enumerate}
\item $\phi(g) \leq g'$.
\item For each interval $[a',b'] \subset [0,g']$, the (restricted) preimage $\phi^{-1}([a',b']) \cap [0, g]$ can be written as a finite disjoint union $\bigcup_i [a^i, b^i]$ of intervals, such that
\[
\rho_g(b^i) \geq \rho_{g'}(b') + \ell \ \ \mbox{~for all }i.
\]
\end{enumerate}
\end{definition}

Those maps were profusely studied in \cite{IKM}. For the reader's convenience we recall a key result, which motivates the above definition and is used in the sequel.

\begin{proposition}[{\cite[Proposition 3.3]{IKM}}]\label{prop:sdep}
	Let  $n, n' \in \NN$, $S = \KK[X_1,\ldots,X_n]$ and $S' = \KK[X_1,\ldots,X_{n'}]$ be two polynomial rings and let $J' \subsetIJ I' \subset S'$ be monomial ideals.
	Consider a monotonic map $\phi: \NN^{n} \to \NN^{n'}$ and
	let $\Phi: S \to S'$ be the map defined by $\Phi(X^a) = X^{\phi(a)}$.
	Set $I := \Phi^{-1}(I')$, $J := \Phi^{-1}(J')$.
	Choose $g \in \NN^n$ and $g' \in \NN^{n'}$, such that every minimal generator of $I$ and $J$ divides $X^g$, and every minimal generator of $I'$ and $J'$ divides $X^{g'}$.
	If $\phi$ changes the Stanley depth by $\ell \in \ZZ$ with respect to $g$ and $g'$, then
	\begin{itemize}
		\item[(i)] $I$ and $J$ are monomial ideals, and
		\item[(ii)] $\sdepth_S\ I/J \geq \sdepth_{S'}\ I'/J' + \ell$.
	\end{itemize}
\end{proposition}

\section{Labelings and lcm-closed sets}\label{sec:lcm}
In this section we present several results on lcm-closed sets that are later needed for the proof of the main results of this paper.
Throughout the section, let $S = \KK[X_1,\dotsc, X_n]$ be a fixed polynomial ring.

A \emph{labeling} of a finite lattice $L$ is a map $w: L \to \Mon{S}$, i.e. an assignment of a monomial to each element of $L$.
Now, for a finite lcm-closed set $G \subset \Mon{S}$ we define a labeling $w_G: \LCMl{G} \to \Mon{S}$ as follows:
For the minimal and maximal elements $\hat{0}, \hat{1}\in G$, we define $w_G(\hat{0}) \defa \gcd\set{p \in G}$ and $w_G(\hat{1}) := 1_\KK$. For every other $m$ we set
\[
w_G(m) \defa \frac{1}{m} \gcd\set{p \in G \with p > m}.
\]
This labeling was introduced in \cite[Eq. (3.3)]{M}.
It satisfies the following inversion formula:
\begin{proposition}\label{prop:inversion}
For $m \in G$ it holds that
\[
m = \prod_{\substack{q \in \LCMl{G} \\ q \ngeq m}} w_G(q).
\]
\end{proposition}
\noindent Note that the formula for $m = \hat{0}$ evaluates to $1_\KK$, but $\hat{0} \neq 1_\KK$.
\begin{proof}
In \cite[Proposition 3.6]{M}, Mapes proves that this formula holds for the minimal elements of $G$, under the additional assumption that $\LCMl{G}$ is atomistic.
However, the argument given there actually shows the formula in the generality claimed here.
\end{proof}

The next corollary gives a characterization of a lcm-closed sets of \emph{squarefree} monomials in terms of $w_G$.
\begin{corollary}\label{cor:squarefree}
	Let $G \subset \Mon{S}$ be a finite lcm-closed set. Then $G$ contains only squarefree monomials if and only if $w_G(m)$ is squarefree for every $m \in \LCMl{G}$ and $\gcd(w_G(m), w_G(m')) = 1_\KK$ for all $m,m' \in \LCMl{G}, m \neq m'$.
\end{corollary}
\begin{proof}
	If the given conditions are satisfied, then all elements of $G$ are squarefree, since by Proposition \ref{prop:inversion} every element of $G$ is a product of different monomials $w_G(m)$ for some $m \in \LCMl{G}$.
	On the other hand, if every element of $G$ is squarefree, then in particular the lcm of all elements is a squarefree monomial. But this is the product of all the $w_G(m)$, so the claimed properties follow.
\end{proof}

We now come to our first key result, which in particular gives a complete description of those pairs $(L, w: L \to \Mon{S})$ that come from a monomial ideal and extends Theorem 3.2 and Proposition 3.6 in \cite{M} to not necessarily atomistic lattices.
\newcommand{\ltoi}[1]{\psi(#1)}
\begin{definition}
	A labeling $w: L \to \Mon{S}$ (on a finite lattice $L$) is \emph{admissible} if it satisfies the following two conditions:
		\begin{enumerate}[ref=\theenumi\emph{\alph*},label=(\alph*)]
			\item\label{aaa} $\gcd(w(a), w(b)) = 1_\KK$ for incomparable $a,b \in L$.
			\item\label{bbb} $w(a) \neq 1_\KK$ if $a \in L$ is meet-irreducible and $w(\hat{1}_L)=1_\KK$.
		\end{enumerate}
\end{definition}
We consider two pairs $(L,w)$ and $(L', w')$ of finite lattices with labelings to be isomorphic if $L \cong L'$ and the isomorphism maps $w$ to $w'$.

\begin{theorem}\label{thm:lcmgeneral}
	The map $G \mapsto (\LCMl{G}, w_G)$ is a bijection between the set of finite lcm-closed sets $G \subset \Mon{S}$, and the set of isomorphism classes of pairs $(L,w: L \to \Mon{S})$ where $L$ is a finite lattice and $w$ is a admissible labeling.
	The inverse map is given by mapping a pair $(L,w)$ to the set
	\begin{equation}\label{eq:lcmgen}
		G(L,w) := \left\{ \prod_{\substack{b \in L \\ b \ngeq a}} w(b) \with a \in L \right\}
	\end{equation}
\end{theorem}
This theorem allows the very simple construction of ideals with a given lcm-lattice. Indeed, for a given lattice $L$ one only needs to choose an admissible labeling. Moreover, considering the possible admissible labelings $w: L \to \Mon{S}$ one gets an overview over the class of all monomial ideals with a fixed lcm-lattice.

\begin{proof}[Proof~of~Theorem~3.4.]
\newcommand{\JC}{\mathcal{J}}
\newcommand{\Lw}{\mathcal{L}}
\newcommand{\Lwa}{\Lw_{\mathrm{ad}}}
	Let $\JC$ denote the set of all finite lcm-closed subsets of $\Mon{S}$ and let $\Lw$ the set of isomorphism classes of pairs $(L,w)$ where $L$ is a finite lattice and $w: L \to \Mon{S}$ is a labeling.
	Moreover, let $\Lwa \subset \Lw$ denote the set where we assume the labeling to be admissible.
	We denote the two maps of the claim by $f: \JC \to \Lw$ and $g: \Lw \to \JC$.
	
	\Cref{prop:inversion} shows that $g \circ f$ is the identity map, so in particular $f$ is injective.
	Therefore, to prove the claim it is sufficient to show that the image of $f$ is $\Lwa$, in other words we need to show that
	\begin{enumerate}
		\item $w_G$ is admissible for each $G \in \JC$, and
		\item for each $(L,w) \in \Lwa$, there exists a $G \in \JC$ with $(L,w) \cong (\LCMl{G}, w_G)$.
	\end{enumerate}
	
	The first item is proven in Lemma 3.4 and Lemma 3.5 of \cite{M}. Again, in \cite{M} the result is only claimed for atomistic lattices, but the argument given there holds in our general setup.
	
	For the second item, our candidate for $G$ is $G(L,w)$ as given by \Cref{eq:lcmgen}.
	So we need to show that $G$ is lcm-closed, $L \cong \LCMl{G(L,w)}$ and (under this isomorphism), $w = w_G$.
	In \cite[Theorem 2.3]{M} it is shown that $L \cong \LCMl{G(L,w)}$ in the case that $L$ is atomistic. However, the argument given there does not directly apply to the general situation, so we need a new proof.
	
	Consider $\psi : L \to G$ defined by $\ltoi{\hat{0}} = 1_\KK$ and
	\[\ltoi{a} \defa \prod_{\substack{b \in L \\ b \ngeq a}} w(b) \]
	for $a \in L, a \neq \hat{0}$.
	The labeling $w$ being admissible implies that for each variable $X_i$ such that $X_i\mid\ltoi{\hat{1}_L}$, the set of $a \in L$ such that $X_i \mid w(a)$ forms a chain $a^i_1 < a^i_2 < \dotsb < a^i_{r_i}$. Set $a^i_{r_i+1} := \hat{1}_L$ for $1 \leq i \leq n$, where $r_i = 0$ if $X_i\nmid\ltoi{\hat{1}_L}$.
	For $a\in L$ let $s(i, a)$ be the minimal index $k$, such that $a^i_k \geq a$.
	As a notation, for a monomial $m \in \Mon{S}$ we define $\ord{i}{m}$ to be the exponent of $X_i$ in $m$.
	We extend this definition to $\LCMl{G}$ by setting $\ord{i}{\hat{0}} := 0$ for all $i$.
	Then
	\begin{equation}\label{eq:ordnung}
		\ord{i}{\ltoi{a}} = \ord{i}{\prod_{j=1}^{s(i,a)-1} w(a^i_j)}.
	\end{equation}
	In particular, if $\ltoi{a} \mid \ltoi{b}$, then $s(i,a) \leq s(i,b)$ for $1 \leq i \leq n$, because of the inequality $\ord{i}{\ltoi{a}} \leq \ord{i}{\ltoi{b}}$.
	
	Let $a, b\in L$. We claim that $a \leq b$ if and only if $\ltoi{a} \mid \ltoi{b}$.
	It is clear from the definition that $a \leq b$ implies $\ltoi{a} \mid \ltoi{b}$, so assume that $\ltoi{a} \mid \ltoi{b}$.
	Every non-maximal element in a finite lattice is the meet of the set of meet-irreducible elements greater than or equal to it.
	So, in order to show $a \leq b$, we may prove the following: Each meet-irreducible element $m$ which is greater than or equal to $b$ is also greater than or equal to $a$.
	So consider such an element $m$. As $w(m) \neq 1_\KK$, there exists an index $i$ such that $X_i \mid w(m)$.

	Then there exists a $k$ such that $m = a^i_{k}$ (where $1\leq k\leq r_i$).
	Now $b \leq m$ implies that $s(i, b) \leq k$.
	But as remarked above, the fact that $\ltoi{a} \mid \ltoi{b}$ implies that $s(i, a) \leq s(i, b) \leq k$, hence $m \geq a$.

	It follows that $\psi$ is injective, as $\ltoi{a} = \ltoi{b}$ implies $\ltoi{a} \mid \ltoi{b} \mid \ltoi{a}$ and thus $a \leq b \leq a$.
	
	Further, we claim that $\ltoi{a\vee b}$ equals the lcm of $\ltoi{a}$ and $\ltoi{b}$.
	For this, first note that $P \geq a \vee b$ if and only if $P \geq a$ and $P \geq b$.
	This implies that $s(i, a \vee b) = \max(s(i,a), s(i, b))$ for all $i$.
	Therefore
	\[\ord{i}{\ltoi{a \vee b}} = \max(\ord{i}{\ltoi{a}}, \ord{i}{\ltoi{b}}) = \ord{i}{\lcm(\ltoi{a}, \ltoi{b})} \]
	for all $i$, hence $\ltoi{a\vee b} = \lcm(\ltoi{a}, \ltoi{b})$.
	
	Summarizing, we have shown that $\psi$ is an injective map $L \to G$ which preserves the join.
	The latter implies that $G$ lcm-closed, i.e. $G = \LCMs{G}$. Hence $\psi$ induces an isomorphism $L \to \LCMl{G}$.
	\medskip
	
	It remains to show that $w(a) = w_G(\ltoi{a})$ for all $a \in L$.
	By definition of $w_G$, we have to show that $\gcd( \ltoi{b} \with b > a ) = \ltoi{a} w(a)$ if $a \neq \hat{0}$ and $\gcd( \ltoi{b} \with b > a ) = w(a)$ for $a = \hat{0}$. We handle both cases together by proving that
	\[
		\ord{i}{\gcd( \ltoi{b} \with b > a )} = \ord{i}{\ltoi{a}} + \ord{i}{w(a)}
	\]
	for each $i$.
	We compute
	\begin{align*}
		\ord{i}{\gcd( \ltoi{b} \with b > a )} &= \min\set{\ord{i}{\ltoi{b}} \with b > a} \\
		&=\min\set{\ord{i}{\prod_{j=1}^{s(i,b)-1} w(a^i_j)} \with b > a} \\
		&=\ord{i}{\prod_{j=1}^{k-1} w(a^i_j)}
	\end{align*}
	where $k := \min\set{s(i,b) \with b > a}$.
	We compute further:
	\begin{align*}
		k = \min(s(i,b) \with b > a) &= \min\set{\min\set{j\with a_j^i \geq b} \with b > a} \\
		&= \min\set{j\with a_j^i > a} \\
		&= \begin{cases}
		s(i,a) + 1 &\text{if } a = a^i_{s(i,a)}, \\
		s(i,a) &\text{otherwise.}
		\end{cases}
	\end{align*}
	Note that in the second case it holds that $\ord{i}{w(a)} = 0$ (otherwise $a = a^i_{s(i,a)}$ because $w$ is admissible).
	
	Recall that $\ord{i}{\ltoi{a}} = \ord{i}{\prod_{j=1}^{s(i,a)-1} w(a^i_j)}$. So we conclude that
	\begin{align*}
	\ord{i}{\gcd( &\ltoi{b} \with b > a )} =\\
	&=\begin{cases}
		\ord{i}{\prod_{j=1}^{s(i,a)-1} w(a^i_j)} + \ord{i}{w( a^i_{s(i,a)})} &\text{if } a = a^i_{s(i,a)},\\
		\ord{i}{\prod_{j=1}^{s(i,a)-1} w(a^i_j)} &\text{otherwise}
	\end{cases} \\
	&=  \ord{i}{\ltoi{a}} + \ord{i}{w(a)}.
	\end{align*}
\end{proof}

\section{Invariants and surjective join-preserving maps}  \label{sec:sur}
This section contains the main results of this paper. They are presented in Subsection \ref{ssec:main} and Subsection \ref{subsec:sur3}. Here we show that the Stanley depth, as well as the usual depth, \emph{are determined} by the lcm-lattice. Subsection \ref{subsec:sur1} contains several related technical results. We end with an example which shows that the $\ZZ$-graded Hilbert depth is not determined by the lcm-lattice.

\subsection{The structure of surjective join-preserving maps} \label{subsec:sur1}
In this Subsection we prove some structural results on surjective join-preserving maps;
these will be needed in the sequel.
The first two structural lemmata will be useful in Subsection \ref{ssec:main}.
\medskip

\newcommand{\pinv}[1]{{#1}^{\dagger}}

	Let $\delta: L \to L'$ be a surjective join-preserving map of finite lattices.
	We define $\pinv{\delta}: L' \to L$ as $\pinv{\delta}(a) \defa \bigvee \delta^{-1}(a)$.
\begin{lemma}\label{lem:pseudoinverse}
	The map $\pinv{\delta}$ has the following properties:
	\begin{enumerate}
		\item $\delta \circ \pinv{\delta}: L' \to L'$ is the identity and $\pinv{\delta}$ is (thus) injective.
		\item $\pinv{\delta}$ is monotonic, i.e. $a \leq b$ implies $\pinv{\delta}(a) \leq \pinv{\delta}(b)$ for $a,b \in L'$.
		\item For $\alpha \in L$ and $b \in L'$, it holds that $\delta(\alpha) \leq b$ if and only if $\alpha \leq \pinv{\delta}(b)$.
	\end{enumerate}
\end{lemma}
\begin{proof}
	The first claim is immediate from the fact that $\delta$ preserves joins.
	For the second, note that $\delta(\pinv{\delta}(a) \vee \pinv{\delta}(b)) = \delta(\pinv{\delta}(a)) \vee \delta(\pinv{\delta}(b)) = a \vee b$ for any $a,b \in L'$.
	Thus $\pinv{\delta}(a) \vee \pinv{\delta}(b)$ is contained in the preimage of $a \vee b$ and hence $\pinv{\delta}(a) \vee \pinv{\delta}(b) \leq \pinv{\delta}(a \vee b)$. Now assume that $a \leq b$. Then
	\[ \pinv{\delta}(a) \leq \pinv{\delta}(a) \vee \pinv{\delta}(b) \leq \pinv{\delta}(a \vee b) = \pinv{\delta}(b). \]
	For the last claim, note that $\alpha \leq \pinv{\delta}(b)$ implies $\delta(\alpha) \leq \delta(\pinv{\delta}(b)) = b$, and $\delta(\alpha) \leq b$ implies $\alpha \leq \pinv{\delta}(\delta(\alpha)) \leq \pinv{\delta}(b)$ (since $\pinv{\delta}$ is monotonic).
\end{proof}

\begin{lemma}\label{lem:pullback}
	Let $G \subset \Mon{S}$ and $G' \subset \Mon{S'}$ be two finite lcm-closed sets of squarefree monomials in two polynomial rings.
	Assume that there exists a surjective join-preserving map $\delta: \LCMl{G} \to \LCMl{G'}$ with $\delta^{-1}(\hat{0}) = \set{\hat{0}}$ such that $\deg w_G(\pinv{\delta}(m')) \geq \deg w_{G'}(m')$ for all $m' \in \LCMl{G'}$.
	Then there exists a ring homomorphism $\Psi: S \to S'$ sending a subset of the variables injectively to the variables of $S'$ and the other variables to $1$. This map satisfies $\Psi(m) = \delta(m)$ for $m \in G$.
\end{lemma}
\begin{proof}
	As $G$ and $G'$ consist only of squarefree monomials, it holds that all values of $w_G$ and $w_{G'}$ are squarefree and pairwise coprime by Corollary \ref{cor:squarefree}.

	We define $\Psi$ as follows:
	For every $m' \in \LCMl{G'}$, choose $\deg w_{G'}(m')$ many variables dividing $w_G(\pinv{\delta}(m'))$ and let them map bijectively to the variables dividing $w_{G'}(m')$. The remaining variables of $S$ are mapped to one.
	By construction, for $m\in \LCMl{G}$ it holds that
	\[ \Psi( w_G(m)) = \begin{cases}
		w_{G'}(m') &\text{~if~} m \in \pinv{\delta}(\LCMl{G'}) \text{~and~} m = \pinv{\delta}(m'); \\
		1 &\text{~if~} m \notin \pinv{\delta}(\LCMl{G'}).
	\end{cases} \]
	Using Proposition \ref{prop:inversion} we conclude that
\[	\Psi(m) = \Psi(\prod_{\substack{q \in \LCMl{G} \\ q \ngeq m}} w_G(q))
		= \prod_{\substack{q \in \LCMl{G} \\ q \ngeq m}} \Psi(w_G(q))
		= \prod_{\substack{q' \in \LCMl{G'} \\ q' \ngeq \delta(m)}} w_{G'}(q') \\
		= \delta(m).
\]
	where $m \in G$. For the third equality, we used part (3) of Lemma \ref{lem:pseudoinverse}.
	Note that the last equality holds because $\delta(m) \neq \hat{0}$.
\end{proof}

The next two structural lemmata will be used in Subsection \ref{ssec:gens}.
\newcommand{\cl}[1]{\overline{#1}}
Fix a meet-irreducible element $a \in L$ and let $a_+ \in L$ denote the unique element covering it.
We consider the equivalence relation $\sim_a$ on $L$ defined by setting $a \sim_a a_+$ and any other element is equivalent only to itself.
\begin{lemma}\label{lem:homom}
	There is a natural lattice structure on $L / \sim_a$, such that the canonical surjection $\pi_a: L \to L/\sim_a$ preserves the join.
	Moreover, if $L$ is atomistic and $a$ is not an atom, then $L / \sim_a$ is atomistic.
\end{lemma}
\begin{proof}
	Let $\cl{b}$ denote the equivalence class of an element $b \in L$. We define $\cl{b} \vee \cl{c} := \cl{b \vee c}$. To show that this is well-defined we have to prove that $b_1 \sim_a b_2$ and $c_1 \sim_a c_2$ implies $b_1 \vee c_1 \sim_a b_2 \vee c_2$.
	For this, we distinguish the cases that either $b_1 = b_2$ or $\set{b_1, b_2} = \set{a, a_+}$ and similarly for $c_1, c_2$. One easily sees that each case is either trivial or follows from the observation that $a\vee b = a_+ \vee b$ for all $b \in L, b \neq a$.
	
	The $\vee$-operation on $L/\sim_a$ inherits associativity, commutativity and idempotency from the join of $L$, cf. \cite[Thm. 2.10]{DP}. Moreover, $L/\sim_a$ inherits a minimal element from $L$, so it is in fact a lattice.
	It is clear that $\pi_a$ preserves this join.
	The last statement is also clear as $\pi_a$ is a bijection on the atoms.
\end{proof}

\begin{lemma}\label{lem:factor}
Let $L, L'$ be finite lattices and $\delta: L \to L'$ a join-preserving map.
\begin{enumerate}
	\item If $\delta$ is not injective, then there exists a meet-irreducible element $a \in L$ such that $\delta(a) = \delta(a_+)$.
	\item If $\delta(a) = \delta(a_+)$ for some meet-irreducible element $a \in L$, then $\delta$ factors through $L / \sim_a$.
\end{enumerate}
\end{lemma}
\begin{proof}
(1) There exists a maximal element $b \in L$ such that the pre-image of $\delta(b)$ has at least two elements, that is $|\delta^{-1}(\delta(b))| > 1$.
	Choose another element $b' \in \delta^{-1}(\delta(b))$,  $b' \neq b$. Then $b' < b$ by maximality, as $\delta(b \vee b') = \delta(b) \vee \delta(b) = \delta(b)$.
	It is easy to see that the interval $[b',b]$ is mapped to $\delta(b)$, so we may choose $a \in L$ such that $\delta(a) = \delta(b)$ and $a$ is covered by $b$.
	We claim that this $a$ is meet-irreducible. Assume to the contrary that there exists another element $c \neq b$ covering $a$. Then
	\[ \delta(c) = \delta(c \vee a) = \delta(c) \vee \delta(a) = \delta(c) \vee \delta(b) =  \delta(c \vee b) \]
	As $b < b \vee c$, it follows from our choice of $b$ that $c = c \vee b$ and thus $c > b$, a contradiction.

(2)	It is clear that $\delta$ factors though $L / \sim_a$ set-theoretically, i.e. there exists a map $\bar{\delta}: L/\sim_a \to L'$ such that $\delta = \bar{\delta} \circ \pi_a$. So we only need to show that $\bar{\delta}$ preserves the join.
	This is an easy computation:
	\[ \bar{\delta}(\cl{b} \vee \cl{c}) = \bar{\delta}(\cl{b \vee c}) = \delta(b \vee c) = \delta(b) \vee \delta(c) = \bar{\delta}(\cl{b}) \vee \bar{\delta}(\cl{c})\]
	for $b,c \in L$.
\end{proof}

\subsection{Stanley projective dimension and surjective join-preserving maps} \label{ssec:main}
In this Subsection we prove the following theorem, which is the main result of this paper.
\begin{theorem}\label{thm:lcmmap}
	Let $H \subsetneq G \subset \Mon{S}$ and $H' \subsetneq G' \subset \Mon{S'}$ be four lcm-closed sets of monomials in two (possibly) different polynomial rings, such that $\<H\> \subsetneq \<G\>, \<H'\> \subsetneq \<G'\>$.
	Assume further that there exists a surjective join-preserving map $\delta: \LCMl{G} \to \LCMl{G'}$ with $\delta^{-1}(\hat{0}) = \set{\hat{0}}$ such that $\delta(H) = H'$.
	Then
	\[ \spdim_S \<G\>/\<H\> \geq \spdim_{S'} \<G'\>/\<H'\>. \]
\end{theorem}

For monomial ideals $J \subsetneq I \subset S$ and $J' \subsetneq I' \subset S'$, the theorem applies in particular to $G := \LCMs{G(I) \cup G(J)}, H := \LCMs{G(J)}, G' := \LCMs{G(I') \cup G(J')}$ and $H' := \LCMs{G(J')}$. However, in general we do not assume that the sets $G, H$ come from \emph{minimal} sets of generators.

A particular case is if $I'$ and $J'$ are the polarizations of $I$ and $J$, respectively.
Therefore, the theorem is a generalization of the authors' result on polarization \cite[Theorem 4.4]{IKM}.
This does not diminish the importance of \cite[Theorem 4.4]{IKM}, since it is required in the proof of Theorem \ref{thm:lcmmap}.
We give a small example to demonstrate that the assumption $\delta^{-1}(\hat{0}) = \set{\hat{0}}$ is necessary.
\begin{example}
	Let $S = \KK[x,y]$, $G = \set{1_\KK,x,y,xy}, G' = \set{x,y,xy}$ and $H = H' = \emptyset$.
	There is a surjective join-preserving map $\delta: \LCMl{G} \to \LCMl{G'}$ defined by mapping $1_\KK$ to $\hat{0}$ and every other element to itself. It holds that $\delta^{-1}(\hat{0}) = \set{\hat{0}, 1_\KK} \neq \set{\hat{0}}$, and indeed the conclusion of \Cref{thm:lcmmap} does not hold:
	\[ \spdim_S \<G\>/\<H\> = \spdim_S S = 0 \ngeq 1 = \spdim_S \<x,y\> = \spdim_S \<G'\>/\<H'\> \]
\end{example}

\bigskip

Before we give the proof of Theorem \ref{thm:lcmmap} we prepare two lemmata.
\begin{lemma}\label{lemma:localize}
	Let $J \subset I \subset S[Y] = \KK[X_1,\dotsc,X_n,Y]$ be two squarefree monomial ideals.
	Let $J' \subset I' \subset S$ be the images of $J$ and $I$ under the map sending $Y$ to $1$.
	Then we have
	\[ \spdim_{S[Y]} I/J \geq \spdim_S I'/J'. \]
\end{lemma}
\noindent This extends \cite[Lemma 2.2]{Ci2}, which shows only the case $J = \<0\>$.
\begin{proof}
	Let $M := I/J$ and let $M_{>0} \subset M$ be the $S[Y]$-submodule of those elements having positive $Y$-degree.
	Every Stanley decomposition of $M$ restricts to a Stanley decomposition of $M_{>0}$, hence $\spdim_{S[Y]} M \geq \spdim_{S[Y]} M_{>0}$.

	On the other hand, we have
	\[ M_{>0} = (I \cap \<Y\>) / (J \cap \<Y\>) \cong (I:Y)/(J:Y) = (I:Y^\infty) / (J:Y^\infty), \]
	where for the last equality we use that $I$ and $J$ are squarefree.
	But $I:Y^\infty = I' \otimes_S S[Y]$ and the same holds for $J$, hence $M_{>0} \cong I'/J' \otimes S[Y]$.
	By \cite[Proposition 5.1]{IJ} we conclude that $\spdim_{S[Y]} M_{>0} = \spdim_S I'/J'$ and the claim follows.
\end{proof}
The second lemma comprises the main part of the proof of the theorem.
\begin{lemma}\label{lemma:goingup}
	Let $H \subsetneq G \subset \Mon{S}$ be two finite lcm-closed sets of squarefree monomials,
	such that $\<H\> \subsetneq \<G\>$.
	Let $m \in \LCMl{G}$ be a fixed element.
	Then there exist two other finite lcm-closed sets of squarefree monomials $H' \subsetneq G' \subset \Mon{S[Y]}$ in one additional variable $Y$, such that the following holds:
	\begin{enumerate}
		\item There is an isomorphism $\delta: \LCMl{G} \to \LCMl{G'}$ of lattices,
			such that $\delta(H) = H'$ and for every $c \in \LCMl{G}$ it holds that
			\[ \deg w_{G'}(\delta(c)) \defa \begin{cases}
				\deg w_{G}(c) &\text{ if } c \neq m,\\
				\deg w_{G}(c) + 1 &\text{ if } c = m.
			\end{cases} \]
		\item $\<H'\> \subsetneq \<G'\> \subset S[Y]$.
		\item $\spdim_{S[Y]} \<G'\>/\<H'\> = \spdim_S \<G\>/\<H\>$.
	\end{enumerate}
\end{lemma}
\begin{proof}
	Consider the map $\tilde{\delta}: \Mon{S} \to \Mon{S[Y]}$ of monomials given by
	\[ \tilde{\delta}(c) = \begin{cases}
		c &\text{ if } c \mid m, \\
		Y \cdot c &\text{ if } c \nmid m.
	\end{cases} \]
	We define $G'$ and $H'$ as the images of $G$ resp.~$H$ under this map.
	It is easy to see that $\tilde{\delta}$ is injective and preserves the lcm of monomials.
	Thus $H' \subsetneq G'$, both sets are lcm-closed, and $\tilde{\delta}$ induces an isomorphism $\delta: \LCMl{G} \to \LCMl{G'}$.
	Moreover, it follows from the definitions that
	\[w_{G'}(m) =  \frac{1}{m} \gcd\set{p \in G' \with p > m} = Y\cdot{}w_{G}(m) \]
	and $w_{G'}(\delta(c)) = w_{G}(c)$ for every other $c \in \LCMl{G}$.
	Part (1) of the lemma is then proven. Part (2) follows straight from the fact that $\delta$ is monotonic.
	
	(3) Let $I := \<G\>, J := \<H\>, I' := \<G'\>$ and $J' := \<H'\>$.
	The inequality ``{$\geq$}'' follows from Lemma \ref{lemma:localize}, as $J$ and $I$ are the images of $J'$ and $I'$ under sending $Y$ to $1$. So we only need to prove the other inequality, which is equivalent to $\sdepth_{S[Y]} I'/J' \geq \sdepth_S I/J + 1$.
	To simplify the notation, we set $0_k \defa (0, \dotsc, 0) \in \NN^k$ and $1_k \defa (1, \dotsc, 1) \in \NN^k$ for $k \in \NN$.
	After relabeling of the variables, we may assume that $m = X_1 X_2 \dotsm X_l$ for some $1 \leq l \leq n$.
	Consider the map $\phi: \NN^{l} \times \NN^{n-l} \times \NN \to \NN^{l} \times \NN^{n-l}$ defined by
	\[
	\phi(h_1,h_2,x) \defa \begin{cases}
		(h_1, 0_{n-l}) &\text{ if } x = 0, \\
		(h_1, h_2) &\text{ if } x > 0.
	\end{cases}
	\]
	It is easy to see that $\phi$ is order preserving, and we claim that it changes the Stanley depth by $1$ with respect to $1_{n+1} \in \NN^{n+1}$ and $1_{n} \in\NN^n$ (see Definition \ref{def:1}).
		
	To prove this claim, it is enough to consider the map $\psi: \NN^{k} \times \NN \to \NN^{k}$, defined by
	\[
	\psi(h,x) \defa \begin{cases}
	0_{k} &\text{ if } x = 0; \\
	h &\text{ if } x > 0.
	\end{cases}
	\]
	where $k := n-l$. It is clear that $\phi=(id_{\NN^{l}}, \psi)$.
	One easily checks that
	\[
	\psi^{-1}(h) = \begin{cases}
	\set{(h,x)\with x > 0} & \text{ if } h \neq 0_k; \\
	\set{(w, 0) \with w \in \NN^{k}} &\text{ if } h = 0_k.
	\end{cases}
	\]
	Consider an interval $[a,b]\subset[0_{k},1_{k}]$. It follows that
	\[
	\psi^{-1}([a,b]) \cap [0_{k+1},1_{k+1}] = \begin{cases}
	[(a,1), (b,1)] & \text{ if } a \neq 0_k; \\
	[(0_k,0), (1_{k},0)] \cup [(0_k,1), (b,1)] & \text{ if } a = 0_k, b \neq 0_k, 1_k; \\
	[(0_k,0), (1_k,1)] & \text{ if } a = 0_k, b = 1_k; \\
	[(0_k,0), (1_k,0))] & \text{ if } a = 0_k, b = 0_k.
	\end{cases}
	\]
	In each case, the Stanley depth is increased (at least) by one.
	The only case that needs a closer look is the second.
	Here, $\rho_{1_{k+1}}((1_k,0)) = k$, but $\rho_{1_{k}}(b) < k$, because $b \neq 1_k$, so this also increases the Stanley depth.
	Therefore $\psi$ increases the Stanley depth (at least) by $1$, and so does $\phi=(id_{\NN^{l}}, \psi)$ (cf.~ \cite[Lemma 3.4]{IKM}).

	Next, let $\Phi: S[Y] \to S$ be map corresponding to $\phi$, i.e. the linear map given on monomials by $\Phi(X^a Y^b) = X^{\phi(a,b)}$.
	We claim that $I' = \Phi^{-1}(I)$, $J' = \Phi^{-1}(J)$.
	Once we have proven this, part (3) follows from \Cref{prop:sdep}.

	For this claim, it suffices to consider $I$ and $I'$. Note that $\Phi(G') = G$ and thus $I' \subset \Phi^{-1}(I)$, so we only need to prove the other inclusion.
	Let $E \subseteq \NN^n$ be the set of exponents of the monomials in $I$.
	Consider a minimal element $e = (h_1,h_2,x) \in \phi^{-1}(E)$.
	By minimality, it follows that $x \in \set{0,1}$. There are two cases:
	\begin{enumerate}
	\item[(i)] If $x = 1$, then $\phi(e) = (h_1,h_2)$. This is clearly contained in $E$ and it is indeed a minimal element of $E$, because if $E$ contains a smaller element $(h'_1,h'_2) < (h_1, h_2)$, then $(h'_1,h'_2, 1) \in \phi^{-1}(h'_1,h'_2) \subset \phi^{-1}(E)$, contradicting the minimality of $e$.
	Hence $X^{(h_1, h_2)} \in G$ and thus $X^{(h_1, h_2)}Y \in I'$.
	\item[(ii)] If $x = 0$, then $h_2 = 0_{n-l}$. Again, $(h_1, 0_{n-l})$ is a minimal element of $E$, because otherwise $(h'_1, 0_{n-l}, 0)$ would be a smaller element in $\phi^{-1}(E)$ as above.
	So again, $X^{(h_1, 0_{n-l})} \in G$ thus $X^{(h_1, 0_{n-l})} \in I'$.
	\end{enumerate}
	So we conclude that $I' \supseteq \Phi^{-1}(I)$ and thus $I' = \Phi^{-1}(I)$.
\end{proof}
\begin{proof}[Proof of \Cref{thm:lcmmap}]
Polarization allows us to replace the sets $G, H$ by lcm-closed sets of squarefree monomials $\tilde{H} \subsetneq \tilde{G}$, such that $\<\tilde{H}\> \subsetneq \<\tilde{G}\>$, $\LCMl{G}$ is isomorphic to $\LCMl{\tilde{G}}$, and the isomorphism restricts to an isomorphism $\LCMl{H} \cong \LCMl{\tilde{H}}$.
Similarly, we may replace $G'$ and $H'$ by lcm-closed sets of squarefree monomials $\tilde{H}' \subsetneq \tilde{G}'$ satisfying the same assumptions.
By composing the map $\delta$ with the given isomorphisms of lattices we also obtain a map $\tilde{\delta}: \LCMl{\tilde{G}} \to \LCMl{\tilde{G}}'$ satisfying the same assumptions as $\delta$.
Moreover, it is an easy corollary of \cite[Theorem 4.3]{IKM} that the Stanley projective dimension is invariant under polarization.
Thus, we may assume that all involved monomial are squarefree.

Next, after repeated application of Lemma \ref{lemma:goingup} to $H$ and $G$ we may also assume that
\[\deg w_{G}(\pinv{\delta}(m)) \geq \deg w_{G'}(m)\]
for all $m \in \LCMl{G'}$. Here, $\pinv{\delta}(m) = \bigvee \delta^{-1}(m)$ as defined in \Cref{subsec:sur1}.
It follows from Lemma \ref{lem:pullback} that $\<G'\>$ and $\<H'\>$ are the images of $\<G\>$ and $\<H\>$ under a homomorphism sending some of the variables to $1$.
Here we use that $\delta(H) = H'$.
Now the claim follows from Lemma \ref{lemma:localize}.
\end{proof}

\subsection{Projective dimension and surjective join-preserving maps} \label{subsec:sur3}
In this subsection, we provide the analogue of Theorem \ref{thm:lcmmap} for the usual projective dimension.
Let $I \subsetneq S$ and $I' \subsetneq S'$ be two monomial ideals, such that there exists a surjective join-preserving map $\delta: \slat_I \to \slat_{I'}$.
If we assume further that $\delta$ is bijective on the generators, then Theorem 3.3 of \cite{GPW} immediately implies that
\begin{equation}\label{eq:obs}
	\pdim_S S/I \geq \pdim_{S'} S'/I'.
\end{equation}
However, we would like to have a result in the same generality as Theorem \ref{thm:lcmmap}. Thus, we point out the following extension of the inequality \eqref{eq:obs}.

\begin{theorem}\label{thm:lcmdepth}
	Under the assumptions of \Cref{thm:lcmmap}, it holds that
	\[ \pdim_S \<G\>/\<H\> \geq \pdim_{S'} \<G'\>/\<H'\>. \]
\end{theorem}

\begin{proof}[Sketch of the proof]
The claim can be proven along the same lines as \cite[Theorem 3.3]{GPW}, therefore we only sketch the necessary modification of the proof given there.

Theorem 3.3 of \cite{GPW} is equivalent to our claim in the case that $H = \LCMs{\tilde{H}}$ for a \emph{minimal} set $\tilde{H}$ of monomial generator of $\<H\>$, $G = \LCMs{\tilde{H} \cup \set{1_\KK}}$,
 and the analogous assumptions on $G'$ and $H'$.

The proof goes by considering the Taylor resolution of $S/\<H\>$. It can be \qq{relabeled} to a resolution of $S'/\<H'\>$, and it is shown that this relabeling maps the minimal free resolution of $S/\<H\>$ to a (generally non-minimal) free resolution of $S'/\<H'\>$.

This proof can be extended to the situation $\<G\> \subseteq S$ by considering the Taylor resolution of $\<G\>/\<H\>$, cf \cite[Def. 3.3.3]{OW}.
One has to consider the Taylor resolution built from the given sets of generators $G, H, G',H'$, not from the minimal ones.

Finally, if $\delta$ is not bijective on atoms, then we replace $G'$ by a multiset: If several elements of $G$ are mapped to the same element of $G'$, then we include several copies of that element in the multiset, one for each preimage.
Consequently, one then considers the Taylor resolution of $\<G'\>/\<H'\>$ with respect to this multiset of generators.
This allows to argue as if $\delta$ were bijective on atoms.
\end{proof}

\begin{corollary}
	Let $J \subsetneq I \subset S$ and $J' \subsetneq I' \subset S'$ be four monomial ideals.
	Assume that $\LCMl{\LCMs{G(I) \cup G(J)}} \cong \LCMl{\LCMs{G(I') \cup G(J')}}$ and that this isomorphism restricts to an isomorphism $L_J \cong L_{J'}$.
	Then it holds that
	\begin{enumerate}
		\item $\spdim_S I/J = \spdim_{S'} I'/J'$;
		\item $\pdim_S I/J = \pdim_{S'} I'/J'$;
		\item $\sdepth_S I/J  - \depth_S I/J = \sdepth_{S'} I'/J' - \depth_{S'} I'/J'$.
	\end{enumerate}
\end{corollary}

In view of part (3) of the preceding corollary, one may ask how the quantity
$\sdepth_S S/I - \depth_S S/I$ behaves under surjective maps of the lcm-lattice.
In general there is no inequality, as can be seen in the following example.

\begin{example}
Consider the maximal ideal $I_1 = \<X_1, \dotsc, X_k\> \subset S = \KK[X_1,\dotsc, X_k]$.
It is well-known that $\sdepth_S S/I_1 = \depth_S S/I_1 = 0$ and thus $\sdepth_S S/I_1 - \depth_S S/I_1 = 0$.

Moreover, consider the ideal $I_2 := \<\frac{x_1\dotsm x_{k}}{x_i} \with 1 \leq i \leq k\> \subset S$.
Its lcm-lattice consists only of $k$ atoms, a maximal element and a minimal element.
Thus every atomistic lattice $L$ with $k$ atoms can be mapped onto it, by mapping atoms to atoms, $\hat{0}$ to $\hat{0}$ and every other element to the maximal element of $L_{I_2}$.
This holds in particular for the lcm-lattice of the ideal $\<x^{k-1}, x^{k-2} y, \dotsc, x y^{k-2}, y^{k-1}\> \subset \KK[x,y]$.
So it follows from Theorem \ref{thm:lcmmap} and Theorem \ref{thm:lcmdepth} that $\spdim_S S/I_2 \leq 2$ and $\pdim_S S/I_2 \leq 2$.
On the other hand, it holds that $\pdim_S S/I_2 \geq 2$ because $I_2$ is not principal and $\spdim_S S/I_2 \geq 2$ by Proposition \ref{prop:boundI} below.
So we can conclude that $\pdim_S S/I_2 = \spdim_S S/I_2 = 2$, hence $\sdepth_S S/I_2 - \depth_S S/I_2 = 0$.

If now $I' \subset S'$ is an arbitrary monomial ideal with $k$ minimal generators, then there are surjective maps $\slat_{I_1} \rightarrow \slat_{I'}$ and $\slat_{I'} \rightarrow \slat_{I_2}$.
Thus if $\sdepth_{S'} S'/I' - \depth_{S'} S'/I' \neq 0$, then this quantity is in general not monotonic under surjective maps.
For example, one may take $I'$ as any monomial ideal whose depth depends on the characteristic of the field.
\end{example}

As a final remark in this section, let us point out that the \emph{$\ZZ$-graded Hilbert depth} is \emph{not} determined by the lcm-lattice.
The $\ZZ$-graded Hilbert depth was introduced by Uliczka \cite{Ul}. It gives an upper bound for both the Stanley depth and the usual depth and is defined as follows:
For a $\ZZ$-graded module $M$ over a standard $\ZZ$-graded polynomial ring $S$, the $\ZZ$-graded Hilbert depth is the maximal depth of any $\ZZ$-graded $S$-module with the same $\ZZ$-graded Hilbert function as $M$.
Usually, the $\ZZ$-graded Hilbert depth is easier to compute than the Stanley depth (see \cite{AP2}).
However, as the following example shows, the analogue of our main result Theorem \ref{thm:lcmmap} does not hold for this invariant:
\begin{example}
	Let $S = \KK[x_1,x_2,x_3,x_4]$ and consider the ideals $I := \<x_1,x_2,x_3,x_4\>$ and $J = \<x_1^3, x_2^2, x_3,x_4\>$ in $S$.
	It is clear that the lcm-lattices of $I$ and $J$ are isomorphic, as both ideals are complete intersections.
	Moreover, the $\ZZ$-graded Hilbert depth of $I$ has been shown to be $4/2 = 2$ in \cite[Example 3.4]{Ul}.
	On the other hand, the Hilbert series of $J$ is
	\[ \frac{2T-T^3-T^4+2T^6-T^7}{(1-T)^4} =  \frac{T^6}{(1-T)^4} + \frac{2T+2T^2+T^3+T^6}{(1-T)^3}.\]
	From the right hand side of the equation we can read off that the $\ZZ$-graded Hilbert depth of $J$ is at least $3$ (\cite[Lemma 2.2]{Ul}).
	Hence the $\ZZ$-graded Hilbert depth is not determined by the isomorphism type of the lcm-lattice.
\end{example}

\section{Applications} \label{sec:applications}
\newcommand{\nice}{invariants-monotone}
Theorems \ref{thm:lcmmap} and \ref{thm:lcmdepth} are the sources for several applications to which this section is devoted.
Essentially all inequalities for the Stanley projective dimension we derive in this section rely on Theorem \ref{thm:lcmmap}.
Using Theorem \ref{thm:lcmdepth} one obtains with the same proof corresponding inequalities for the usual projective dimension.
More generally, this holds for any invariant of an ideal or its lcm-lattice which satisfies the conclusion of Theorem \ref{thm:lcmmap}. These are, for example, cardinality, length, width, breadth, order dimension and interval dimension of the lcm-lattice.

\smallskip
We will use Theorem \ref{thm:lcmmap} several times in this section, so it seems convenient to introduce a name for the maps satisfying its hypotheses. So we call a map $\delta: L \to L'$ between finite lattices \emph{\nice{}} if it is join-preserving, surjective, and $\delta^{-1}(\hat{0}) = \set{\hat{0}}$.

\newcommand{\bool}[1]{\mathfrak{B}(#1)}
\newcommand{\sbool}[1]{\overline{\mathfrak{B}}(#1)}
\newcommand{\maxk}{\mathfrak{m}}

\subsection{Bounds for the Stanley depth in terms of generators}\label{ssec:gens}
For $k \in \NN$ let $\bool{k}$ denote the boolean lattice on $k$ atoms, i.e. the lattice of subsets of a $k$-element set. Note that $\bool{k}$ is the lcm-lattice of any ideal generated by $k$ variables.
\begin{remark}\label{lem:free}
	For every atomistic lattice $L$ on $k$ atoms, there exists an \nice{} map $\delta: \bool{k} \to L$, which may be constructed as follows. Let $\delta$ map the atoms of $\bool{k}$ bijectively on the atoms of $L$ and set $\delta(\hat{0}) = \hat{0}$. For every other element $a \in \bool{k}$ we set $\delta(a) := \delta(a_1) \vee \dotsb \vee \delta(a_l)$ where $a = a_1 \vee a_2 \vee \dotsb \vee a_l$ is the unique (up to order) way to write $a$ as a join of atoms.
\end{remark}

First, we give a uniform proof of important results previously obtained by several authors providing bounds on the Stanley depth.
\begin{proposition}\label{prop:boundI}
	Let $k > 1$ and let $I \subset S$ be a monomial ideal with $k$ minimal generators. Let further $\maxk_k := (Y_1, \dotsc, Y_k) \subset S_k := \KK[Y_1, \dotsc, Y_k]$ be the monomial maximal ideal on $k$ generators. Then the following inequalities hold:
\begin{enumerate}
	\item $1 \leq \spdim_S I \leq \spdim_{S_k} \maxk_k = \lfloor\frac{k}{2}\rfloor$
	\item $2 \leq \spdim_S S/I \leq \spdim_{S_k} S_k/\maxk_k = k$
\end{enumerate}
Moreover, if $I$ is a complete intersection, then the upper bounds are attained.
\end{proposition}
The assumption that $k >1$ was introduced in order to avoid the zero module $S/S$.
The upper bound for $\spdim_S S/I$ was originally proven by Cimpoea\c{s} \cite[Prop.~1.2]{Ci} and the upper bound for $\spdim_S I$ was originally proven by Okazaki \cite{O}, resp. in the squarefree case by Keller and Young \cite{KY}.
For a complete intersection $I$ the Stanley depth of  $I$  was originally determined by Shen \cite{S} and of $S/I$ by Rauf \cite{R2}.
\begin{proof}[Proof~of~Proposition~\ref{prop:boundI}.]
	The Stanley depth of the maximal ideal $\maxk_k$ was computed by Bir{\'o} et al.~in \cite{BHKTY}. Moreover, the Stanley depth of $\KK=S_k / \maxk_k$ is zero. Thus the values of the upper bounds are known.
	The inequalities $\spdim_S I \leq \spdim_{S_k} \maxk_k$ and $\spdim_S S/I \leq \spdim_{S_k} S/\maxk_k$ follow from Theorem \ref{thm:lcmmap}, since by Remark \ref{lem:free} there exists an \nice{} map $\slat_{\maxk_k} = \bool{k} \to L_I$.

	If $I$ is a complete intersection, then $\slat_I \cong \slat_{\maxk_k}$, therefore $\spdim_S I = \spdim_{S_k} \maxk_k$ and $\spdim_S S/I = \spdim_{S_k} S/\maxk_k$ by Theorem \ref{thm:lcmmap}.

	For the lower bound, note that every ideal in $n$ variables with more than one minimal generator has Stanley depth less than $n$.
	Moreover, if $I$ has at least $2$ generators, then there exists a monomial $m$ in $S \setminus I$ and two coprime monomials $n_1, n_2 \in S$ such that $m n_1, m n_2 \in I$.
 This implies that $S/I$ has an associated prime of height at least $2$ and thus the Stanley depth of $S/I$ is at most $n - 2$.
\end{proof}

In the next result we characterize the case of equality for the upper bound in part (2) of Proposition \ref{prop:boundI}.
Recall that monomial complete intersections can be characterized as those ideals $I$ whose number of generators equals the projective dimension of $S/I$.
In this sense, the last sentence of the following theorem extends the result on the Stanley depth of complete intersections \cite{S}.
\begin{theorem}
	Let $I \subset S$ be a monomial ideal with $k > 1$ minimal generators. Then the following are equivalent:
	\begin{enumerate}
		\item $\slat_I \cong \bool{k}$, i.e. $I$ has the lcm-lattice of a complete intersection.
		\item $\pdim_S S/I = k$.
		\item $\spdim_S S/I = k$.
	\end{enumerate}
	Moreover, if $\pdim_S S/I = k-1$ then both $S/I$ and $I$ satisfy Stanley's inequality \eqref{conj:stanley}.
\end{theorem}
\begin{proof}
	If $\slat_I \cong \bool{k}$ then $\spdim_S S/I = k$ by Proposition \ref{prop:boundI}.
	Moreover, in this situation $\pdim_S S/I = k$ because this is the projective dimension of a $k$-generated complete intersection.
	So we need to show that $\slat_I \ncong \bool{k}$ implies that $\spdim_S S/I \leq k-1$ and $\pdim_S S/I \leq k-1$.
	
	By Remark \ref{lem:free} there exists an \nice{} map $\bool{k} \to \slat_I$.
	As $\slat_I \ncong \bool{k}$ this map is not injective, so by Lemma \ref{lem:factor} it factors through $\bool{k} / \sim_a$ for some meet-irreducible element $a \in \bool{k}$.
	But the automorphism group of $\bool{k}$ acts transitively on the set of meet-irreducible elements, so $\slat := \bool{k} / \sim_a$ does not depend on $a$.
	If $J$ is a monomial ideal (in some polynomial ring $S'$) whose lcm-lattice equals $\slat$, then
	the Theorems \ref{thm:lcmmap} and \ref{thm:lcmdepth} imply that it suffices to prove $\spdim_{S'} S'/J \leq k-1$ and $\pdim_{S'} S'/J \leq k-1$.
	
	We claim that we can choose
	\[
		J = \<x_1^2, \dotsc, x_{k-1}^2, x_1 x_2 \dotsm x_{k-1}\> \subset \KK[x_1,\dotsc,x_{k-1}].
	\]
	To see this, let us identify each element of $\bool{k}$ by the set of atoms below it.
	Then---up to an automorphism---we have $a = \set{1, \dotsc, k-1}$.
	The meet-irreducible elements of $\slat$ are the $(k-1)$-subsets of $[k] := \set{1, \dotsc,k}$ other than $a$, and the $(k-2)$-subsets of $a$.
	
	We can choose a labeling $w$ as follows:
	Set $w([k] \setminus \{i\}) = w(a \setminus \{i\}) = x_i$ and all other elements of $L$ are mapped to $1$.
	This labeling is admissible, so by Theorem \ref{thm:lcmgeneral} the corresponding ideal has the desired lcm-lattice. Moreover, it is easy to see that the generators of this ideal are as claimed.
	Here, $x_i^2$ corresponds to the atom $\set{i}$ for $1 \leq i \leq k-1$ and $x_1 x_2 \dotsm x_{k-1}$ corresponds to $\set{k}$.
	
	Note that $J$ is a monomial ideal in $k-1$ variables, which implies that the claimed inequalities $\spdim_{S'} S'/J \leq k-1$ and $\pdim_{S'} S'/J \leq k-1$ hold trivially.
	
	Now we turn to the last statement of the theorem. Assume that $\pdim_S S/I = k-1$.
	We already showed that this implies that $\spdim_S S/I \neq k$ and thus $\spdim_S S/I \leq k - 1 = \pdim_S S/I$.
	Further, it holds that $\pdim_S I = k-2$, and by the argument above it suffices to show that $\pdim_{S'} J \leq k-2$.
	For this, we note that \cite[Theorem 27]{H} implies that $\sdepth_{S'} J > 0$ and hence $\spdim_{S'} J \leq k-2$.
\end{proof}

\subsection{Deformations of monomial ideals}\label{ssec:deform}
The notion of \emph{deformation} of a monomial ideal was introduced by Bayer et al.~\cite{BPS} and further developed in Miller et al.~\cite{MSY}.
In order to include the case of quotients $I/J$, we slightly extend the definition found in \cite{MSY}. Recall that $\ord{j}{m}$ denotes the exponent of $X_j$ in a monomial $m \in \Mon{S}$.
\begin{definition} (1) Let $G \subset \Mon{S}$ be a finite set of monomials.
	A \emph{deformation} of $G$ is a set of vectors $\ve_g = (\ve^g_1, \dotsc, \ve^g_n) \in \NN^n$ for $g \in G$
	subject to the following conditions:
	\begin{equation}\label{eq:deform}
	\begin{aligned}
	\ord{j}{g} > \ord{j}{h} &\implies \ord{j}{g} + \ve^g_j > \ord{j}{h} + \ve^h_j \quad\text{ and }\\
	\ord{j}{g} = 0 &\implies \ve^g_j = 0.
	\end{aligned}
	\end{equation}
	
	(2) Let $J \subsetneq I \subset S$ be two monomial ideals with (not necessarily minimal) generating sets $G_I$ and $G_J$.
	Let further $\ve$ be a deformation of the union $G_I \cup G_J$.
	We set $G_I(\ve) := \set{g\cdot \xf^{\ve_g} \with g \in G_I}$ and $G_J(\ve)$ is defined analogously.
	Then we call the two ideals $I_\ve := \<G_I(\ve)\>$ and $J_\ve := \<G_J(\ve)\>$ a \emph{common deformation} of $I$ and $J$.
	Note that the condition \eqref{eq:deform} implies that $J_\ve \subsetneq I_\ve$.
\end{definition}

\begin{proposition}\label{prop:gen}
	Let $J \subsetneq I \subset S$ be two monomial ideals and let $\gen{J} \subsetneq \gen{I} \subset S$ be a common deformation of $I$ and $J$. Then
	$\sdepth_S I/J \geq \sdepth_S \gen{I} / \gen{J}$
	and the same holds for the usual depth.
\end{proposition}
\begin{proof}
	As noticed in \cite{GPW}, the map sending each deformed monomial $g \cdot x^{\ve_g}$ to the corresponding original monomial $g$ induces an \nice{} map $\LCMl{\LCMs{G_I(\ve) \cup G_J(\ve)}} \to \LCMl{\LCMs{G_I \cup G_J}}$, so the claim follows from Theorem \ref{thm:lcmmap} and Theorem \ref{thm:lcmdepth}, respectively.
\end{proof}

The most important deformations are the generic deformations. Let us recall the definition from \cite{MSY}.
\begin{definition} (1) A monomial $m \in S$ is said to \emph{strictly divide} another monomial $m' \in S$ if $m \mid \frac{m'}{x_i}$ for each variable $x_i$ dividing $m'$.

(2) A monomial ideal $I \subset S$ is called \emph{generic} if for any two minimal generators $m,m'$ of $I$ having the same degree in some variable, there exists a third minimal generator $m''$ that strictly divides $\lcm(m,m')$.

(3) A deformation of a monomial ideal $I$ is called \emph{generic} if the deformed ideal $\gen{I}$ is generic.
\end{definition}

\begin{corollary}
	If $I \subset S$ is a monomial ideal such that $\depth_S S/I = \depth_S S/\gen{I}$ for some generic deformation of $I$, then $\sdepth_S S/I \geq \depth_S S/I$ (i.~e.~Stanley conjecture holds for $S/I$).
\end{corollary}
\begin{proof}
	It was proven by Apel in \cite{A2} that $\sdepth_S S/J \geq \depth_S S/J$ for every generic monomial ideal $J$. So the claim follows from Proposition \ref{prop:gen} by considering the generic deformation of $I$.
\end{proof}

\subsection{Colon ideals and associated primes}
In this subsection we consider colon ideals with respect to monomials.
Both results of this section can be proven directly, but we would like to illustrate that they also follow from our main result.
Moreover, our proof works uniformly for both depth and Stanley depth.
The first result was originally proven by Seyed Fakhari in \cite[Proposition 2.5]{SF}.

\begin{proposition}\label{prop:colon}
	Let $J \subsetIJ I \subset S$ be two monomial ideals and let $v \in \Mon{S}$ be a monomial.
	Then
	\[ \spdim_S I/J \geq \spdim (I:v)/(J:v) \]
	and the same holds for the projective dimension.
\end{proposition}
\begin{proof}

Let $L := \LCMs{G(I) \cup G(J)}$ and $L' := \LCMs{G(J)}$.
We consider the map
\[ \delta': L \rightarrow \Mon{S},  \quad m \mapsto \frac{m \vee v}{v}. \]
It is easy to see that $\delta'$ preserves the join.
We will show that the image $\delta'(L)$ generates $I:v$ and similarly $\delta'(L')$ generates $J:v$.
Then we can extend $\delta'$ to an \nice{} map $\delta: \LCMl{L} \to \LCMl{\delta'(L)}$ by setting $\delta(\hat{0}) := \hat{0}$, so our claim follows from Theorem \ref{thm:lcmmap} (resp.~Theorem \ref{thm:lcmdepth}).

By symmetry, we only consider $L$. It is clear that $\delta'(L) \subseteq I:v$. For the other inclusion consider a monomial $m \in I:v$. Then there exists a generator $g$ of $I$ such that $g \mid v m$. Hence $g \vee v \mid v m$ and thus $\delta'(g) \mid m$. So $I:v$ is contained in the ideal generated by the image of $\delta'$.
\end{proof}

As a consequence, we get the well-known bound on the depth and Stanley depth in terms of the height of associated prime ideals, see \cite[Theorem 9]{H}.

\newcommand{\pp}{\mathfrak{p}}
\begin{corollary}\label{cor:ass}
Let $I \subset S$ be a monomial ideal. If $I$ has an associated prime $\pp \subset S$ of height $p$, then
\begin{align*}
\spdim_S S/I, \pdim_S S/I &\geq p;\\
\spdim_S I &\geq \lfloor\frac{p}{2}\rfloor;\\
\pdim_S I &\geq p-1.
\end{align*}
\end{corollary}
\begin{proof}
This follows from the foregoing proposition, given the known values of $\spdim$ and $\pdim$ for monomial prime ideals.
\end{proof}

From the proof of Proposition \ref{prop:colon} one can also extract the following lattice-theoretical statement, which might be of independent interest. As we do not use it, we omit the proof.
\begin{proposition}
	Let $L$ be a finite atomistic lattice and let $p \in \NN$. The following are equivalent:
	\begin{enumerate}
		\item There exists an \nice{} map $L \rightarrow \bool{p}$ onto the boolean lattice on $p$ generators.
		\item There exists a monomial ideal $I$ with $L \cong L_{I}$ and $I$ has an associated prime of height $p$.
	\end{enumerate}
\end{proposition}

\subsection{Ideals generated in a single degree}
In this subsection we show that every finite atomistic lattice can be realized as lcm-lattice of a monomial ideal, whose minimal generators all have the same degree.
\begin{lemma}\label{lemma:singledegree}
	Let $L$ be a finite lattice and let $A \subset L$ be an antichain, i.e. a set of pairwise incomparable elements.
	Then there exists an lcm-closed set of monomials $G \subset \Mon{S}$ such that $L \cong \LCMl{G}$ and the monomials in $G$ corresponding to the elements of $A$ all have the same degree.
\end{lemma}
\begin{proof}
	First, choose an admissible labeling $w_1: L \rightarrow \Mon{S}$, where $S$ is some polynomial ring with sufficiently many variables.
	If the monomials $m_a = \prod_{b \ngeq a} w_1(b)$ for $a \in A$ (as in \eqref{eq:lcmgen}) have all the same degree, then we are already done.
	Otherwise, let $a_1, a_2, \dotsc, a_r \in A$ be the elements whose monomials have the maximum degree among the monomials corresponding to $A$.
	We modify our $w_1$ by setting
	\[ w_2(m) :=
	\begin{cases}
		X_i w_1(m) &\text{ if } m = a_i,\\
		w_1(m) &\text{ otherwise.}
	\end{cases} \]
	where the $X_i$ are new variables.
	As $A$ is an antichain, it follows from \eqref{eq:lcmgen} that the degree of the monomials corresponding to $a_1, a_2, \dotsc, a_r \in A$ under $w_2$ increases by $r-1$, while the degree of all other monomials corresponding to $A$ increases by $r$.
	Hence after iterating this procedure finitely many times, all monomials corresponding to $A$ have the same degree.
\end{proof}
\begin{proposition}\label{prop:singledegree}
	Let $J \subsetIJ I \subset S$ be two monomial ideals.
	Then one can find monomial ideals $J' \subsetIJ I' \subset S'$ with $\LCMl{\LCMs{G(I) \cup G(J)}} \cong \LCMl{\LCMs{G(I') \cup G(J')}}$ where the isomorphism maps $L_J$ to $L_{J'}$,
	such that one's choice of the following holds:
	\begin{enumerate}
	\item Either $I'$ is generated in a single degree, or
	\item $J'$ is generated in a single degree.
	\end{enumerate}
\end{proposition}
\begin{proof}
	The set of minimal generators of $I$ forms an antichain in $\LCMl{\LCMs{G(I) \cup G(J)}}$. Applying the foregoing Lemma \ref{lemma:singledegree} to it yields ideals $J' \subset I'$ with the same lcm-lattice, where $I'$ is generated in a single degree.
	On the other hand, applying the lemma to the antichain formed by the minimal generators of $J$ results in $J'$ being generated in a single degree.
\end{proof}

\begin{remark}
Note that our construction does in general not allow to assume that both ideals are generated in a single degree.
However, if $S=I$ (or more generally if one ideal is principal), then this is possible.
\end{remark}

\subsection{Further applications}\label{sec:future}
Finally, let us point out several results that also follow from our main result.
The following result was originally proven in Ishaq \cite{Is}. See also Seyed Fakhari \cite{SF} for a different proof.
\begin{proposition}
	Let $J \subsetneq I \subset S$ be two monomial ideals. Then
	\[ \sdepth_S \sqrt{I}/\sqrt{J} \geq \sdepth_S I/J \]
	and the same holds for the usual depth.
\end{proposition}
\begin{proof}
\newcommand{\rad}[1]{\sqrt{#1}}
	For a monomial $m$ we write $\rad{m}$ for the product of the variables dividing $m$. Define $G'(I)$ as $\set{\rad{m} \with m \in G(I)}$ and $G'(J)$ similarly. Then $G'(I)$ and $G'(J)$ generate $\sqrt{I}$ and $\sqrt{J}$.
	Moreover, the map $m \mapsto \rad{m}$ gives rise to an \nice{} map $\LCMl{\LCMs{G(I) \cup G(J)}} \to \LCMl{\LCMs{G'(I) \cup G'(J)}}$. So the claim follows from Theorem \ref{thm:lcmmap} and Theorem \ref{thm:lcmdepth}.
\end{proof}

\begin{remark} There are several operations known on monomial ideals which do not change the lcm-lattice. In all these cases, the effect on depth and Stanley depth can be deduced from Theorem \ref{thm:lcmmap} and Theorem \ref{thm:lcmdepth}. We give references to several articles where these operations were previously studied:
	\begin{itemize}
		\item[--] Polarization. However, we used this result (cf.~\cite{IKM}) in our proof of Theorem \ref{thm:lcmmap}.
		\item[--] Multiplication of an ideal by a monomial \cite[Theorem 1.4]{Ci}.
		\item[--] The quotient modulo a non-zerodivisor monomial \cite[Theorem 1.1]{R2}, and the extension of the polynomial ring by new variables \cite[Prop. 5.1]{IJ}.
		\item[--] The constructions in Propositions 5.1 and 5.2 of \cite{IKM}, which themselves are extensions of other results (Ishaq and Qureshi \cite[Lemma 2.1]{IQ}, \cite[Lemma 1.1]{Ci2}, \cite[Lemma 2.3]{S}).
        \item[--] Corollary~3.3.~in Yanagawa~\cite{Y} and the properties treated in Anwar and Popescu \cite{AP}.
	\end{itemize}
\end{remark}

\section{The Stanley projective dimension as an invariant of lattices?} \label{sec:Stanley}
In view of our main result, the Stanley projective dimension of an ideal depends only on its lcm-lattice.
Hence one can interpret this number as a combinatorial invariant of the lattice itself.
Let us make this precise.
\newcommand{\I}{\mathrm{I}}
\newcommand{\Q}{\mathrm{Q}}
\begin{definition}\label{def:inv}
	Let $L$ be an finite atomistic lattice. Choose an ideal $I \subset S = \KK[X_1,\dotsc,X_n]$ such that $L_I \cong L$.
	We define
	\begin{enumerate}
	\item $\pdim_\I L \defa \pdim I$,
	\item $\pdim_\Q L \defa \pdim S/I$,
	\item $\spdim_\I L \defa \spdim I$ and
	\item $\spdim_\Q L \defa \spdim S/I.$
	\end{enumerate}
	Here, the subscripts $\Q$ and $\I$ stand for \qq{quotient} and \qq{ideal}. In particular, the subscript $\I$ is not the name of the ideal $I$ involved in the definition.
\end{definition}

Note that it trivially holds that $\pdim_\I L = \pdim_\Q L - 1$ and that these invariants may depend on the underlying field $\KK$.
On the other hand, $\spdim_\I L$ and $\spdim_\Q L$ clearly do not depend on the field (cf. \cite[Remark 3.5]{IKM4}).
Now, \cite[Conjecture 2]{A1},\cite[Conjecture 1]{A2} and \cite[Conjecture 64]{H} can be formulated for lattices:
\begin{conjecture}\label{conj:stanleylattice}
	For all finite atomistic lattices $L$, it holds that
	\begin{enumerate}
		\item $\spdim_\I L \leq \pdim_\I L$,
		\item $\spdim_\Q L \leq \pdim_\Q L$, and
		\item $\spdim_\I L \leq \spdim_\Q L - 1$.
	\end{enumerate}
\end{conjecture}
After a first version of this paper was posted on the arXiv, Duval, Goeckner, Klivans, and Martin \cite{counterexample} found a counterexample to the Stanley conjecture, which also disproves part (2) of this conjecture
\footnote{We refrain from deleting part (2) of this conjecture, because it is cited in the already published article \cite{IKM3}.}.
Parts (1) and (3) are still open.
\begin{remark}
Part (1) of Conjecture \ref{conj:stanleylattice} depends implicitly on the underlying field $\KK$.
However, as it follows from Hochster's formula \cite[Corollary 5.12]{millersturm} that $\pdim S_\ZZ/I \otimes \mathbb{Q} \leq \pdim S_\ZZ/I \otimes \ZZ/p\ZZ$ for any monomial ideal $I \subset S_\ZZ := \ZZ[X_1,\dotsc,X_n]$ over the integers and any prime $p$, the characteristic $0$ case of the conjecture implies it for all fields.
\end{remark}

A natural question is how these new invariants relate to the usual invariants of lattices. Proposition \ref{prop:boundI} can be interpreted in this way,
where the number of generators $k$ corresponds to the width of the subposet of join-irreducible elements of $L$.
Another step in this direction was taken by the second author together with S.A. Seyed Fakhari in \cite{KSF}, where we show that $\spdim_\I L$ and $\spdim_\Q L$ are bounded above by the length and by the order dimension of $L$.

One problem with the second part of Definition \ref{def:inv} is that one has to choose an ideal $I$.
Remark that the invariants $\pdim_\I$ and $\pdim_\Q$ may be computed directly from $L$: It follows from the results in \cite{GPW} that $\pdim_\Q L = \pdim_\I L + 1 = \min\set{i \with \tilde{H}_i(\Delta((\hat{0}, m)_L); \KK) \neq 0,\ m \in L} + 2$, where $\Delta((\hat{0}, m)_L)$ denotes the order complex of the open interval $(\hat{0}, m)_L  \subset L$.
Motivated by this we pose the following question:
\begin{question}
Is there a purely lattice theoretic description of $\spdim_\I L$ and $\spdim_\Q L$?
\end{question}

\section*{Acknowledgments}

The authors are greatly indebted to Volkmar Welker for pointing out the guidelines leading to Theorem \ref{thm:lcmdepth}. We also wish to thank Winfried Bruns and the anonymous reviewers for several helpful suggestions.

\bibliographystyle{alpha}
\bibliography{LCM_Final}

\end{document}